\documentclass{article}
\usepackage{amsmath}
\usepackage{amssymb}
\usepackage{amsthm}
\usepackage{bbm,mathtools}

\newcommand{\op}[1]{\prescript{o}{}{#1}}
\newcommand{\pp}[1]{\prescript{p}{}{#1}}

\def\graph{\mathop{\rm gph}}

\newcommand{\midb}{\;\middle|\;}
\newcommand{\one}{\mathbbm 1}

\def\reals{\mathbb{R}}

\def\ereals{\overline{\mathbb{R}}}

\def\comp{\raise 1pt \hbox{$\scriptstyle\circ$}}

\def\dom{\mathop{\rm dom}\nolimits}

\def\upto{{\raise 1pt \hbox{$\scriptstyle \,\nearrow\,$}}}
\def\downto{{\raise 1pt \hbox{$\scriptstyle \,\searrow\,$}}}

\def\epi{\mathop{\rm epi}}
\def\hypo{\mathop{\rm hypo}}

\def\FF{(\F_t)_{t\ge 0}}

\def\B{{\cal B}}

\def\cD{{\textrm(C)}}
\def\cP{{\textrm(C_p)}}

\def\F{{\cal F}}

\def\G{{\cal G}}
\def\H{{\cal H}}

\def\O{{\cal O}}
\def\P{{\cal P}}

\def\S{{\cal S}}

\def\T{{\cal T}}

\newtheorem{theorem}{Theorem}
\newtheorem{lemma}[theorem]{Lemma}
\newtheorem{corollary}[theorem]{Corollary}

\newtheorem{example}{Example}

\theoremstyle{definition}
\newtheorem{definition}{Definition}

\begin{document}
\title{Optional and predictable projections of normal integrands and convex-valued processes}
\author{Matti Kiiski\footnote{Aalto University, Department of Mathematics and Systems Analysis, P.O. Box 11100, FI-00076 Aalto, Finland ({\tt matti.kiiski@gmail.com})} \and Ari-Pekka Perkki\"o\footnote{Technische Universit\"at Berlin, Department of Mathematics, Building MA,
Str. des 17. Juni 136, 10623 Berlin, Germany ({\tt perkkioe@math.tu-berlin.de}). Corresponding author. The author is grateful to the Einstein Foundation for the financial support.}}
\maketitle

\begin{abstract}
This article studies optional and predictable projections of integrands and convex-valued stochastic processes. The existence and uniqueness are shown under general conditions that are analogous to those for conditional expectations of integrands and random sets. In the convex case, duality correspondences between the projections and projections of epigraphs are given. These results are used to study projections of set-valued integrands. Consistently with the general theory of stochastic processes, projections are not constructed using reference measures on the optional and predictable sigma-algebras.
\end{abstract}


\section{Introduction}

Normal integrands, set-valued integrands, and set-valued mappings have proven to be fundamental concepts both in discrete and continuous time optimization. For applications to discrete time stochastic optimization and mathematical finance, see, e.g. \cite{rw83,pp10}. As to the continuous time, these concepts are important in optimal control and calculus of variations \cite{roc78,mor6b}. In stochastic optimal control, normal integrands are already used in \cite{bis73b} whereas set-valued integrands and set-valued mappings appear in stochastic differential inclusions and set-valued stochastic integrals \cite{kis13}. In mathematical finance, set-valued mappings are used to model, e.g., portfolio constraints \cite{cs12} and currency markets \cite{ks9}.

In discrete time, conditional expectations of normal integrands and random sets have received considerable attention \cite{bis73,cv77,thi81,tru91,chs3}. In stochastic optimization, the dynamic programming equations can be given in terms of conditional expectations of integrands \cite{evs76,ppr16}. For applications to ergodic theory and statistics, see \cite{chs3}. 

This article extends the analysis to the continuous time setting by studying optional and predictable projections of normal integrands and set-valued processes. We will define these projections so that they correspond to conditional expectation of integrands and random sets. When a set-valued process is single-valued or a normal integrand is simply a stochastic process, our definitions reduce to the ordinary optional and predictable projections. Thus our definitions are consistent with the general theory of stochastic processes. An application to stochastic optimal control is given in \cite{pp15a} whereas financial applications will be studied elsewhere. 

In Section \ref{sec:nisp}, we review the definitions and basic properties of normal integrands and set-valued processes; a systematic treatment can be found from \cite{rw98}. In Section~\ref{sec:st}, we prove versions of optional and predictable section theorems that involve graphs of optional and predictable set-valued processes. In Sections~\ref{sec:proj} and \ref{sec:projni}, we use these section theorems to prove the existence and uniqueness of optional and predictable projections of integrands. In particular, we give general conditions under which projections of normal integrands are again normal integrands. These are similar to those in \cite{chs3} where conditional expectations of normal integrands were studied. However, when specialized to conditional expectations of integrands on $\reals^d$, our results slightly extend both \cite[Theorem 2.1]{chs3} and \cite[Theorem 1.4]{tru91}; see Example~\ref{ex:ce}.

In Section~\ref{sec:svsp}, we prove the existence and uniqueness of optional and predictable projections of convex-valued stochastic processes. The methods in \cite{hu77} and \cite{wan01}, that deal with conditional expectations of general random sets, use a given probability measure on $\Omega$; here we do not have such measure on $\Omega\times\reals_+$ at our disposal. Instead, our method is based on conjugacy arguments and on our existence results for projections of integrands. We emphasize that we do not construct projections of set-valued processes as conditional expectations with respect to a given measure on $\Omega\times\reals_+$. This is the case, e.g., in the proof of \cite[Theorem 3.7]{wan01b}. 

We finish the article by applying our main results to study projections of set-valued integrands in Section~\ref{sec:svi}. We give existence and uniqueness results for such projections as corollaries of our main theorems. To this end, we assume that the integrands are convex-valued and that they have appropriate inner and outer semicontinuity properties.

\section{Normal integrands and set-valued processes}\label{sec:nisp}

We assume throughout that the filtered probability space $(\Omega,\F,\FF,P)$ satisfies the usual hypotheses, that is, $\FF$ is right continuous and $\F_0$ contains all the $P$-null sets. We denote the predictable and optional $\sigma$-algebras on $\Omega\times\reals_+$ by $\P$ and $\O$. The set of $[0,\infty]$-valued stopping times is denoted by $\T$ and its subset of predictable times by $\T_p$. We use the common notations $\F_\tau=\{A\in\F\mid A\cap\{\tau\le t\}\in\F_t\ \forall t\in\reals_+\}$ and $\F_{\tau-}=\F_0\vee\sigma\{A\cap\{t<\tau\}\mid A\in\F_t,\, t\in\reals_+\}$. Standard references for these basic concepts are, e.g., \cite{dm78,dm82,hwy92}.

Let $\S$ be a $\sigma$-algebra on $\Omega\times\reals_+$. We say that an extended real-valued function $h:\Omega\times\reals_+\times\reals^d\to\ereals$ is an {\em $\S$-integrand} on $\reals^d$ if it is $\S\otimes\B(\reals^d)$-measurable. We assume throughout that all integrands are $\F\otimes \B(\reals_+)$-integrands on $\reals^d$ unless stated otherwise. An integrand $h$ can be viewed as a measurable stochastic process $(\omega,t)\mapsto h_t(x,\omega)$ depending on the parameter $x\in\reals^d$, or, as a function-valued stochastic process.

Recall that a set on $\Omega\times\reals_+$ is {\em evanescent} if its projection to $\Omega$ is a $P$-null set. For integrands $h^1$ and $h^2$, we denote $h^1\le h^2$ if
\[
\{(\omega,t)\in\Omega\times\reals_+\mid \exists x\in\reals^d: h^1_t(x,\omega)>h^2_t(x,\omega)\}
\]
is evanescent and we say that $h^1$ and $h^2$ are {\em indistinguishable} if $h^1=h^2$. Later, we will not distinguish two indistinguishable integrands, i.e., they are regarded as the same. Accordingly, all properties of integrands are understood to be satisfied outside an evanescent set. For example, we say that an integrand $h$ is {\em convex} if $h_t(\cdot,\omega)$ is a convex function outside an evanescent set. Likewise, all equalities involving (at most countable number of) integrands are understood to hold everywhere outside some evanescent set. Note that, since $\F_0$ is complete, such evanescent set can always be assumed to be predictable. Indeed, for evanescent $E\subseteq\Omega\times\reals_+$, we have $E\subseteq\pi(E)\times\reals_+$ and $\pi(E)\times\reals_+\in\P$ by \cite[Theorem III.21]{hwy92}. Here $\pi$ denotes the projection from $\Omega\times\reals_+$ onto $\Omega$.

Throughout, for an integrand $h$ and an arbitrary $w:\Omega\times\reals_+\rightarrow\reals^d$, we denote by $h(w)$ the process $(\omega,t)\mapsto h_t(w_t(\omega),\omega)$. Note that $h(w)$ is $\S$-measurable whenever $h$ is an $\S$-integrand and $w$ is $\S$-measurable. For measurable $\phi:\Omega\rightarrow\ereals$, we use the convention that the expectation
\[
E \phi :=\int \phi(\omega)dP(\omega)
\] 
is $+\infty$ unless the positive part is integrable. In particular, sums of extended real numbers are defined as $+\infty$ if any of the terms equals $+\infty$.

We call set-valued mappings from $\Omega\times\reals_+$ to $\reals^d$ set-valued stochastic processes. A set-valued stochastic process $\Gamma$ is $\S$-{\em measurable} if the {\em inverse image} 
\[
\Gamma^{-1}(O)=\{(\omega,t)\in\Omega\times\reals_+ \mid \Gamma_t(\omega)\cap O\ne\emptyset\}
\]
of every open $O\subseteq \reals^d$ belongs to $\S$. In particular, the {\em domain mapping} 
\[
\dom \Gamma :=\{(\omega,t)\mid \Gamma_t(\omega)\neq \emptyset\}
\]
is $\S$-measurable whenever $\Gamma$ is so. The set
\[
\graph \Gamma:=\left\{(\omega,t,x)\in\Omega\times\reals_+\times\reals^d \midb x\in \Gamma_t(\omega)\right\}
\]
is known as the {\em graph} of $\Gamma$. Throughout, set-valued stochastic processes are considered to be equal if they coincide outside an evanescent set. We call a process $w$ a {\em selection} of $\Gamma$ if $w_t(\omega)\in\Gamma_t(\omega)$ outside an evanescent set. We denote the set of $\S$-measurable selections of $\Gamma$ by $L^0(\S;\Gamma)$. When $\S=\F\otimes\B(\reals_+)$, we write simply $L^0(\Gamma)$.

An extended real-valued function $h:\Omega\times\reals_+\times\reals^d\to\ereals$ is said to be a {\em normal $\S$-integrand} on $\reals^d$ if the {\em epi-graphical mapping}
\[
(\omega,t)\mapsto\epi h_t(\cdot,\omega)=\{(x,\alpha)\in \reals^d\times\reals|\, h_t(x,\omega)\leq\alpha\}
\]
is closed-valued and measurable. A normal $\S$-integrand $h$ is always $\S\otimes\B(\reals^d)$-measurable \cite[Corollary 14.34]{rw98}.

When $\S$ is complete with respect to some $\sigma$-finite measure, then an $\S$-integrand $h$ for which $h_t(\cdot,\omega)$ is lower semicontinuous for all $(\omega,t)$, is a normal $\S$-integrand \cite[Corollary 14.34]{rw98}. Some authors take these properties as the definition of a normal integrand, but since we will work with incomplete $\sigma$-algebras, we use the more precise concept given in terms of the measurability of the epigraphical mapping. Indeed, this is the case with, e.g., the optional and predictable $\sigma$-algebras; see, e.g., \cite{ran90}. 

\section{Optional and predictable section theorems}\label{sec:st}

For a set $A$, we denote $\one_A(x)=1$ if $x\in A$ and $\one_A(x)=0$ otherwise, whereas $\delta_A(x)=0$ if $x\in A$ and $\delta_A(x)=+\infty$ otherwise. For a set $A$ in $\Omega\times\reals_+$, we denote by $1_A$ the stochastic process $(1_A)_t(\omega)=\one_A(\omega,t)$. Motivated by the notion of $P$-discretely dense set (see \cite[Section 1.2]{tru91}), a set $\H$ of $\S$-measurable processes on $\Omega\times\reals_+$ is said to be {\em $\S$-discretely dense} if, for every $\S$-measurable process $w$, there exists an $\S$-measurable covering $(A^\nu)_{\nu=1}^\infty$ of $\Omega\times\reals_+$ and a sequence $(w^\nu)_{\nu=1}^\infty$ in $\H$ such that $1_{A^\nu}w=1_{A^\nu}w^\nu$ for every $\nu$. For example, the set of bounded optional processes is $\O$-discretely dense. Recall that a stochastic process $w$ is said to be {\em bounded} if there is an $M>0$ such that $|w_t(\omega)|\le M$ outside an evanescent set.

The following section theorems for set-valued processes will play an important role. They reduce to the optional and predictable section theorems (see e.g. \cite{dm78,hwy92}) in the special case when $\Gamma=\reals^d$ on $\dom\Gamma$.  For $\sigma:\Omega\rightarrow \reals_+\cup\{+\infty\}$, we denote $\graph{\sigma}=\{(\omega,t)\in\Omega\times\reals_+\mid \sigma(\omega)=t\}$. We denote by $\pi$ the projection from $\Omega\times\reals_+$ to $\Omega$.

\begin{theorem}\label{thm:ost}
Assume that $\graph{\Gamma}$ of a set-valued process $\Gamma$ is $\O\otimes\B(\reals^d)$-measurable and that $\H$ is an $\O$-discretely dense set. For any $\epsilon>0$, there exists $\tau\in\T$ and $w\in\H$ such that 
\begin{align*}
\graph{\tau}&\subseteq\dom\Gamma,\\
w_\tau&\in\Gamma_\tau\text{ on }\{\tau<\infty\},\\
P(\{\tau<\infty\})&\ge P(\pi(\dom\Gamma))-\epsilon.
\end{align*}
\end{theorem}
\begin{proof}
By \cite[Theorem III.44]{dm78}, there is a measurable $\sigma:\Omega\rightarrow \reals_+\cup\{+\infty\}$ such that $P(\{\sigma<\infty\})=P(\pi(\dom\Gamma))$ and $\graph{\sigma}\subseteq \dom\Gamma$. We define a set function on $\Omega\times\reals_+$ by
\[
I(C)=\inf_B\left\{ E\left[(1_B)_\sigma\one_{\{\sigma<\infty\}}\right]\midb C\subset B, B\in \O\right\}.
\]
It is easy to verify that $B\mapsto E\left[(1_B)_\sigma\one_{\{\sigma<\infty\}}\right]$ is a measure on $\O$, so $I$ is an $\O$-capacity, by \cite[Remark~2.8]{et13}. By \cite[Theorem~1.32 and Theorem~1.35]{hwy92}, $\dom\Gamma$ is $\O$-capacitable (see \cite[Definition 1.33]{hwy92}), so there is an $A\in\O$ with $A\subset\dom\Gamma$ and $I(A)\ge I(\dom\Gamma)-\epsilon/4$, i.e.,  
\[
E \left[(1_A)_\sigma\one_{\{\sigma<\infty\}}\right]\ge P(\pi(\dom\Gamma))-\epsilon/4.
\]
By the optional section theorem \cite[Theorem IV.84]{dm78}, there exists $\tilde\tau\in\T$ such that $P(\{\tilde\tau<\infty\})\ge P(\pi(A))-\epsilon/4$ and $\graph{\tilde\tau}\subseteq A$.

Let $\mu$ be a measure on $\Omega\times\reals_+$ defined by $\mu(F)=E[(1_F)_{\tilde\tau} \one_{\{\tilde\tau<\infty\}}]$. The set $\graph{\Gamma}\cap (A\times\reals^d)$ is $\O\otimes\B(\reals^d)$-measurable, so, by a measurable selection theorem \cite[Theorem III.22]{cv77}, there exists $\O^\mu$-measurable $\hat w$ such that $\hat w_{\tilde\tau}\in\Gamma_{\tilde\tau}$ on $A$, where $\O^\mu$ is the $\mu$-completion of $\O$. Therefore, there exists $\O$-measurable $\check w$ and  $\tilde C\in\O$ such that $\check w=\hat w$ on $\tilde C$, $\tilde C\subset A$, and such that $\mu(\tilde C)=\mu(A)$. Here $\mu(A)=P(\{\tilde\tau<\infty\})$. Since $\mu$ is countably additive and bounded, and since the sequence $(A^\nu)_{\nu=1}^\infty$ in the definition of $\O$-discretely dense can be chosen increasing, there exists $A^\nu\in\O$ and $w^\nu\in\H$ such that $1_{A^\nu}\check w=1_{A^\nu}w^\nu$ and $\mu(A^\nu\cap\tilde C)\ge\mu(\tilde C)-\epsilon/4$. We denote $w=w^\nu$.

By the optional section theorem, there exists $\tau\in\T$ such that $\graph{\tau}\subset (A^\nu\cap\tilde C)$ and $P(\{\tau<\infty\})\ge P(\pi(A^\nu\cap\tilde C))-\epsilon/4$. We have $\graph{\tau}\subseteq\dom\Gamma$, $w_\tau\in\Gamma_\tau$ on $\{\tau<\infty\}$ and
\begin{align*}
P(\{\tau<\infty\})&\ge P(\pi(A^\nu\cap\tilde C))-\epsilon/4 \ge \mu(\tilde C)-2\epsilon/4\\
&\ge P(\pi(A))-3\epsilon/4 \ge E \left[(1_A)_{\sigma}\one_{\{\sigma<\infty\}}\right] -3\epsilon/4\\
&\ge P(\pi(\dom\Gamma))-\epsilon. 
\end{align*}
\end{proof}
 
\begin{theorem}\label{thm:pst}
Assume that $\graph{\Gamma}$ of a set-valued process $\Gamma$ is $\P\otimes\B(\reals^d)$-measurable and that $\H$ is a $\P$-discretely dense set. For any $\epsilon>0$, there exists $\tau\in\T_p$ and $w\in\H$ such that 
\begin{align*}
\graph{\tau}&\subseteq\dom\Gamma,\\
w_\tau&\in\Gamma_\tau\text{ on }\{\tau<\infty\}\\
P(\{\tau<\infty\})&\ge P(\pi(\dom\Gamma))-\epsilon.
\end{align*}
\end{theorem}
\begin{proof}
This can be proved like Theorem~\ref{thm:ost} by invoking the predictable cross section theorem \cite[Theorem IV.85]{dm78} instead of the optional cross section theorem.
\end{proof}

Given set-valued processes $\Gamma$ and $\tilde\Gamma$, $\Gamma$ is said to be smaller than $\tilde\Gamma$ if $\Gamma_t(\omega)\subseteq\tilde\Gamma_t(\omega)$ outside an evanescent set, in which case we denote $\Gamma\subseteq\tilde\Gamma$. For set-valued mappings $\tilde S$ and $S$ from $\Omega$ to $\reals^d$, we denote $\tilde S\subseteq S$ if $\tilde S(\omega)\subseteq S(\omega)$ almost surely. Such mappings are called {\em random sets}. It follows from the Castaing representation \cite[Theorem 14.5]{rw98}, that if $\Gamma$ is optional, then $\Gamma_\tau$ is an $\F_\tau$-measurable random closed set for each $\tau\in\T$. Likewise, if $\Gamma$ is predictable, then $\Gamma_\tau$ is $\F_{\tau-}$-measurable random closed set for each $\tau\in\T_p$.

The following result gives the set-valued analog of the fundamental result (see, e.g., \cite[Theorem 4.10]{hwy92}) that optional (resp. predictable) stochastic processes $v^1$ and $v^2$ satisfy $v^1\le v^2$ if and only $v^1_\tau \le v^2_\tau$ almost surely for every bounded $\tau\in\T$ (resp. for every bounded $\tau\in\T_p$). A stopping time $\tau$ is said to be {\em bounded} if there is an $M>0$ such that $\tau(\omega)\le M$ almost surely.
\begin{lemma}\label{lem:inc}
Let $\tilde\Gamma$ and $\Gamma$ be optional set-valued processes. Then $\tilde\Gamma\subseteq\Gamma$ if and only if $\tilde\Gamma_\tau\subseteq\Gamma_\tau$  for every bounded $\tau\in\T$. If $\tilde\Gamma$ and $\Gamma$ are predictable, then it is sufficient that $\tilde\Gamma_\tau\subseteq \Gamma_\tau$ for every bounded $\tau\in\T_p$.
\end{lemma}
\begin{proof}
The necessity is obvious. Conversely, assume for a contradiction that $\tilde\Gamma$ is not smaller that $\Gamma$, i.e., the domain of $\tilde\Gamma\backslash\Gamma$ is nonevanescent. Since
\[
\graph{(\tilde\Gamma\backslash\Gamma)}=\graph{\tilde\Gamma}\backslash\graph\Gamma 
\]
is $\O\otimes\B(\reals^d)$-measurable, there exists, by Theorem~\ref{thm:ost}, $\tilde\tau\in\T$ and optional process $w$ such that $P(\{\tilde\tau<\infty\})>0$, $w_{\tilde\tau}\in\tilde\Gamma_{\tilde\tau}$ on $\{\tilde\tau<\infty\}$, and $w_{\tilde\tau}\notin\Gamma_{\tilde\tau}$ on $\{\tilde\tau<\infty\}$. Defining $\tau:=\tilde\tau\wedge M$ for $M$ large enough, we get a contradiction. The latter claim is proved similarly using Theorem~\ref{thm:pst}.
\end{proof}

\section{Projections of integrands}\label{sec:proj}

If an $\reals^d$-valued stochastic process $w$ is {\em $\T$-integrable} in the sense that $\one_{\{\tau<\infty\}}|w_\tau|$ is integrable for every $\tau\in\T$, then there exists an $\reals^d$-valued optional process $\op w$ such that
\begin{align*}
\op{w}_\tau\one_{\{\tau<\infty\}}&=E\left[w_\tau\one_{\{\tau<\infty\}}\mid \F_\tau\right]\quad P\text{-a.s. for all $\tau\in\T$}
\end{align*}
which is unique up to indistinguishability \cite[Theorems 5.1]{hwy92}. Similarly, if an $\reals^d$-valued stochastic process $w$ is {\em $\T_p$-integrable} in the sense that $\one_{\{\tau<\infty\}}|w_\tau|$ is integrable for every $\tau\in\T_p$, then there exists an $\reals^d$-valued predictable process $\pp w$ such that
\begin{align*}
\pp{w}_\tau\one_{\{\tau<\infty\}}&=E\left[w_\tau\one_{\{\tau<\infty\}}\mid \F_{\tau-}\right]\quad P\text{-a.s. for all $\tau\in\T_p$}
\end{align*}
which is unique up to indistinguishability \cite[Theorems 5.3]{hwy92}.  The processes $\op{w}$ and $\pp w$ are known as the {\em optional projection} and the {\em predictable projection} of $v$, respectively.

Likewise, when $v$ is an extended real-valued nonnegative stochastic process, there exists an optional process $\op v$ and a predictable process $\pp v$ such that
\begin{align*}
\op{v}_\tau\one_{\{\tau<\infty\}}&=E\left[v_\tau\one_{\{\tau<\infty\}}\mid \F_\tau\right]\quad P\text{-a.s. for all $\tau\in\T$},\\
\pp{v}_\tau\one_{\{\tau<\infty\}}&=E\left[v_\tau\one_{\{\tau<\infty\}}\mid \F_{\tau-}\right]\quad P\text{-a.s. for all $\tau\in\T_p$},
\end{align*}
which are unique up to indistinguishability. For nonnegative real-valued processes, this is a classical fact \cite[Theorem VI.43]{dm82} which extends to the extended nonnegative real-valued case using the monotone convergence theorem. For an extended real-valued stochastic process $v$ such that $v^+$ or $v^-$ is $\T$-integrable, we set $\op v= \op (v^+) - \op (v^-)$, where $v^+=\max\{v,0\}$ and $v^-=\max\{-v,0\}$. Correspondingly, for an extended real-valued stochastic process $v$ such that $v^+$ or $v^-$ is $\T_p$-integrable, we set $\pp v= \pp (v^+) - \pp (v^-)$.

The following definitions extend these basic concepts to integrands. We assume throughout that $h$ is an integrand on $\reals^d$. We define $h^+=\max\{h,0\}$ and $h^-=\max\{-h,0\}$, and denote by $\Lambda^\O_h$ the set of $\reals^d$-valued optional processes $w$ for which $h^+(w)$ or $h^-(w)$ is $\T$-integrable. Similarly, $\Lambda^\P_h$ denotes the set of $\reals^d$-valued predictable processes $w$ for which $h^+(w)$ or $h^-(w)$ is $\T_p$-integrable. We note that $h^+(w)=h(w)^+$ and $h^-(w)=h(w)^-$ for any process $w$, and we use these notations interchangeably depending on which one we find more natural in the context. 

\begin{definition}
For an integrand $h$, we say that ${^oh}:\Omega\times\reals_+\times\reals^d\rightarrow\ereals$ is an {\em optional projection} of $h$ if $^oh$ is an $\O$-integrand on $\reals^d$ such that
\[
[\op h](w)=\op[h(w)]\quad\forall\ w\in\Lambda^\O_h.
\]
\end{definition}

\begin{definition}
For an integrand $h$, we say that $\pp h:\Omega\times\reals_+\times\reals^d\rightarrow\ereals$ is a {\em predictable projection} of $h$ if $\pp h$ is a $\P$-integrand on $\reals^d$ such that
\[
[\pp h](w)=\pp [h(w)]\quad\forall\ w\in\Lambda^\P_h.
\]
\end{definition}

\begin{example}
When $h_t(x,\omega)=v_t(\omega)$ for some $\T$-integrable stochastic process $v$, we simply have $\op h=\op v$. Thus the definition of an optional projection of $h$ reduces to that of an optional projection of the stochastic process $v$.
\end{example}

The following lemma shows that optional and predictable projections of $h$ are characterized by any $\O$ or $\P$-discretely dense subset $\H$ of $\Lambda^\O_h$ and $\Lambda^P_h$, respectively.
\begin{lemma}\label{lem:dds}
Let $h$ be an integrand. If there exists an $\O$-discretely dense set $\H$ contained in $\Lambda^\O_h$ and an optional integrand $\tilde h$ such that
\[
\tilde h(w)=\op [h(w)]\quad\forall\ w\in\H,
\]
then $\tilde h$ is an optional projection of $h$. If there exists a $\P$-discretely dense set $\hat\H$ contained in $\Lambda^\O_h$ and a predictable integrand $\hat h$ such that 
\[
\hat h(w)=\pp [h(w)]\quad\forall\ w\in\hat \H,
\]
then $\hat h$ is a predictable projection of $h$.
\end{lemma}
\begin{proof}
Assume for a contradiction that $\tilde h(w)\ne\op [h(w)]$ for some $w\in\Lambda^\O_h$. Applying Theorem~\ref{thm:ost} to the set-valued process 
\[
\Gamma_t(\omega)=\begin{cases}
\{w_t(\omega)\}\quad&\text{if } (\tilde h(w))_t(\omega)\ne(\op [h(w)])_t(\omega)\\
\emptyset\quad&\text{otherwise},
\end{cases}
\]
we get $\tau\in\T$ with $P(\tau<\infty)>0$ and $\tilde w\in\H$ such that  $\one_{\{\tau<\infty\}}\tilde h(\tilde w)_\tau\ne\one_{\{\tau<\infty\}}\op [h(\tilde w)]_\tau$, which is a contradiction. The predictable case is proved similarly using Theorem~\ref{thm:pst}.
\end{proof}

The following lemma states that the projections preserve order. Choosing $h^1=h^2$ in the lemma, one gets uniqueness of the projections as soon as $\Lambda^\O_h$ contains an $\O$-discretely dense set. The example after the proof shows that this is not the case if such $\O$-discretely dense set does not exist.

\begin{lemma}\label{lem:cp1}
For integrands $h^1$ and $h^2$ with $h^1\le h^2$, we have ${\op h^1}\le {\op h^2}$ and ${\pp h^1}\le {\pp h^2}$ whenever the projections exist and $\Lambda^{\O}_{h^1}\cap \Lambda^\O_{h^2}$ contains an $\O$-discretely dense set or $\Lambda^{\P}_{h^1}\cap \Lambda^\P_{h^2}$ contains a $\P$-discretely dense, respectively.
\end{lemma}
\begin{proof}
We prove only the optional case, the claim for ${^ph^1}\le {^ph^2}$ follows similarly from Theorem~\ref{thm:pst}. We assume for a contradiction that $\dom\Gamma$ is not evanescent for
\[
\Gamma_t(\omega)=\{x\in\reals^d\mid {^oh}^1_t(x,\omega)>\alpha>{^oh}^2_t(x,\omega)\}
\]
for some $\alpha\in\reals$. By Theorem~\ref{thm:ost}, there is a $\tau\in\T$ with $P(\{\tau<\infty\})>0$ and $w\in(\Lambda^{\O}_{h^1}\cap \Lambda^\O_{h^2})$ with $w_\tau\in\Gamma_\tau$ on $\{\tau<\infty\}$. Thus
\begin{align*}
E \left[h^1_{\tau}(w_{\tau})\one_{\{{\tau}<\infty\}}\right]&=E \left[[{^o h}^1]_{\tau}(w_{\tau})\one_{\{{\tau}<\infty\}}\right]>\alpha P(\{{\tau}<\infty\})\\
&>E\left[[{^o h}^2]_{\tau}(w_{\tau})\one_{\{{\tau}<\infty\}}\right] =E\left[h^2_{\tau}(w_{\tau})\one_{\{{\tau}<\infty\}}\right],
\end{align*}
which is a contradiction with $h^1\le h^2$. Therefore, $\dom\Gamma$ is evanescent, and hence ${^oh^1}\le {^oh^2}$.
\end{proof}

\begin{example}
Let $\eta$ be a measurable process for which $\eta^+$ and $\eta^-$ are not $\T$-integrable, and let
\[
h_t(x,\omega)=\eta_t(\omega).
\]
Here $\Lambda^\O_h=\emptyset$, so every optional integrand is an optional projection of $h$.
\end{example} 

The following result is a monotone convergence theorem for projections of integrands. 
\begin{theorem}\label{thm:mon}
Let $(h^\nu)_{\nu=1}^\infty$ be a nondecreasing sequence of integrands and 
\[
h=\sup_\nu h^\nu.
\]
If each $\op h^\nu$ exists and there exists an $\O$-discretely dense subset of optional processes $w$ for which $h^1(w)^-$ is $\T$-integrable, then $\op h$ exists and 
\[
\op h=\sup_\nu \op h^\nu.
\]
If each $\pp h^\nu$ exists and there exists a $\P$-discretely dense subset of predictable processes $w$ for which $h^1(w)^-$ is $\T_p$-integrable, then $\pp h$ exists and 
\[
\pp h=\sup_\nu \pp h^\nu.
\]
\end{theorem}
\begin{proof}
Let $\H$ be an $\O$-discretely dense set such that $h^1(w)^-$ is $\T$-integrable for each $w\in\H$. Then $\H$ is an $\O$-discretely dense subset of each $\Lambda^\O_{h^i}$ and $\Lambda^\O_h$. For any $w\in\H$, $\tau\in\T$ and $A\in\F_\tau$,  the monotone convergence and Lemma~\ref{lem:cp1} imply that
\begin{align*}
E[\one_A\one_{\{\tau<\infty\}}{h}_\tau(w_\tau)] &= E[\one_A\one_{\{\tau<\infty\}}\sup_\nu h^\nu_\tau(w_\tau)]\\
&= \sup_\nu E[\one_A\one_{\{\tau<\infty\}}h^\nu_\tau(w_\tau)]\\
&= \sup_\nu E[\one_A\one_{\{\tau<\infty\}}\op h^\nu_\tau(w_\tau)]\\
&=  E[\one_A\one_{\{\tau<\infty\}}\sup_\nu \op h^\nu_\tau(w_\tau)].
\end{align*}
Thus $\sup_\nu \op h^\nu=\op h$ by Lemma~\ref{lem:dds}. The predictable case is similar.
\end{proof}

Now we are ready to prove our first main result.
\begin{theorem}\label{thm:cp}
If $\Lambda^\O_h$ contains an $\O$-discretely dense set, then $h$ has a unique optional projection. If $\Lambda^\P_h$ contains a $\P$-discretely dense set, then $h$ has a unique predictable projection.
\end{theorem} 
\begin{proof}
The uniqueness follows from Lemma~\ref{lem:cp1}, so it suffices to prove the existence. We use a monotone class argument together with Lemma~\ref{lem:dds}.

Let $A\in\B(\reals^d)$ and $B\in \F\otimes \B(\reals_+)$. Then $\op(\one_A 1_B)=\one_A \op 1_B$. Indeed, for every $\tau\in\T$, $F\in\F_\tau$ and optional process $w$, we get from \cite[Corollary 3.23]{hwy92} that $\one_A(w_\tau\one_{\{\tau<\infty\}})$ is $\F_\tau$-measurable, so
\[
\begin{split}
E[\one_F\one_A(w_\tau)(1_B)_\tau\one_{\{\tau<\infty\}}]&=E[\one_F\one_A(w_\tau\one_{\{\tau<\infty\}})(1_B)_\tau\one_{\{\tau<\infty\}}]\\
&=E[\one_F\one_A(w_\tau\one_{\{\tau<\infty\}})(\op 1_B)_\tau\one_{\{\tau<\infty\}}]\\
&=E[\one_F\one_A(w_\tau)(\op 1_B)_\tau\one_{\{\tau<\infty\}}].\\
\end{split}
\]

Let $(h^\nu)_{\nu=1}^\infty$ be a nondecreasing nonnegative sequence which converges point-wise to some bounded $h$ and such that ${^o h^\nu}$ exists for every $\nu$.  By Theorem~\ref{thm:mon}, $\sup_\nu \op h^\nu_t$ is the optional projection of $h$. Evidently, when a bounded $h$ has an optional projection, then $-h$ has an optional projection as well, so, by the monotone class theorem \cite[Theorem 1.4]{hwy92}, every bounded integrand admits an optional projection.

Any nonnegative integrand is a point-wise limit of a nondecreasing sequence of nonnegative bounded integrands, so, the existence of an optional projection for such integrand follows from Theorem~\ref{thm:mon}.

Assume that $\H$ is an $\O$-discretely dense set contained in $\Lambda^\O_h$. For $w\in\H$, either $h^+(w)$ or $h^-(w)$ is $\T$-integrable, so either $\op h^+(w)$ or $\op h^-(w)$ is integrable as well. Thus
\[
\begin{split}
E[\one_Fh_\tau(w_\tau)\one_{\{\tau<\infty\}}]&=E[\one_Fh^+_\tau(w_\tau)\one_{\{\tau<\infty\}}]-E[\one_Fh^-_\tau(w_\tau)\one_{\{\tau<\infty\}}]\\
&=E[\one_F\op h^+_\tau(w_\tau)\one_{\{\tau<\infty\}}]-E[\one_F\op h^-_\tau(w_\tau)\one_{\{\tau<\infty\}}]\\
&=E[\one_F(\op h^+_\tau(w_\tau)-\op h^-_\tau(w_\tau))\one_{\{\tau<\infty\}}],
\end{split}
\]
so the optional projection of $h$ is given by $\op h=\op [h^+]-\op [h^-]$. The predictable case is proved similarly.
\end{proof}

We finish this section by relating the optional and predictable projection to the conditional expectation. Let $f:\Omega\times\reals^d\rightarrow \ereals$ be an extended real-valued function that is $\F\otimes\B(\reals^d)$-measurable, i.e., $f$ is an $\F$-integrand on $\reals^d$. Let $\G$ be a sub-$\sigma$-algebra of $\F$. The function $E^\G f:\Omega\times\reals^d:\rightarrow\ereals$ is called the {\em conditional expectation} of $f$ if $E^\G f$ is a $\G$-integrand on $\reals^d$ such that
\begin{equation}\label{eq:ce}
(E^\G f)(\eta) = E^\G f(\eta)\quad\forall \eta\in \Lambda^\G_f,
\end{equation}
where $\Lambda^\G_f$ is the set of $\G$-measurable random variables $\eta$ for which $f(\eta)^+$ or $f(\eta)^-$ is integrable. Here and in what follows, $E^\G$ denotes the conditional expectation with respect to $\G$. For results on conditional expectations, we refer to \cite{bis73,cv77,thi81,tru91,chs3}.
\begin{example}
When $h_t(x,\omega)=f(x,\omega)$ and $\F_t=\G\subseteq\F$ for all $t$, then the definition of an optional projection of $h$ reduces to the definition of a $\G$-conditional expectation of $f$.
\end{example}
The example above together with Theorem~\ref{thm:cp} gives the following corollary. 
\begin{corollary}\label{cor:ece}
Let $f$ be an $\F$-integrand on $\reals^d$. If $\Lambda^\G_f$ contains a $\G$-discretely dense set, then there exists a unique $\G$-conditional expectation of $f$.
\end{corollary}

For $\F$-integrands on $\reals^d$, Corollary~\ref{cor:ece} generalizes the existence and uniqueness results of both \cite[Theorem 2.1]{chs3} and \cite[Theorem 1.4]{tru91}. This can be seen from the following example which does not satisfy the assumptions of either theorem. Indeed, $f$ is not bounded from below by an integrable random variable on any set that intersects $\reals_-$, whereas, for $\alpha > 0$, $f$ is not lower semicontinuous.
\begin{example}\label{ex:ce}
Assume that $d=1$, $\G$ is the trivial (and complete) $\sigma$-algebra, and that 
\[
f(x,\omega)=
\begin{cases}
x\eta(\omega)\quad&\text{if }x\ne 0,\\
\alpha\quad&\text{if } x=0,
\end{cases}
\]
where $\eta$ is $\F$-measurable, nonnegative, and non-integrable, and $\alpha\in\reals$. Here it is easy to verify that the assumptions of Corollary~\ref{cor:ece} are met and that
\[
E^\G f(x,\omega)=
\begin{cases}
-\infty\quad&\text{if } x<0,\\
\alpha\quad&\text{if } x=0,\\
+\infty\quad&\text{if }x>0.
\end{cases}
\]
\end{example}

\section{Projections of normal integrands}\label{sec:projni}

The aim of this section is to give general conditions under which the projections of a normal integrand exist and they are normal integrands. Choosing $\alpha\le 0$ in Example~\ref{ex:ce}, we see that this is not the case in general. 

The following definition is motivated by the results in \cite{chs3}, where conditional expectations of normal integrands were studied.
\begin{definition}
An integrand $h$ is of {\em class $\cD$} if there exists a sequence of open sets $(B^i)_{i=1}^\infty$ such that $\bigcup_{i=1}^\infty B^i=\reals^d$ and, for every $i$, there exists a $\T$-integrable $m^i$ such that 
\[
h\one_{B^i}\ge m^i\one_{B^i}\text{ outside an evanescent set.}
\]
If $m^i$ can be chosen $\T_p$-integrable, then $h$ is of {\em class $\cP$}.
\end{definition}
We remark that each $B^i$ in the definition of class $\cD$ and $\cP$ can be chosen bounded and the sequence $(B^i)_{i=1}^\infty$ increasing. 
For an integrand of class $\cD$, the set of bounded optional process is an $\O$-discretely dense set contained in $\Lambda^\O_h$. Similarly, when $h$ is of class $\cP$, the set of bounded predictable processes is a $\P$-discretely dense set contained in $\Lambda^\P_h$. The following lemma follows directly from Lemma~\ref{lem:dds} and Lemma~\ref{lem:cp1}.
\begin{lemma}\label{lem:bddc}
If $h$ is an integrand of class $\cD$ and there exists an optional integrand $\tilde h$ such that
\[
\tilde h(w)=\op [h(w)]
\]
for all bounded optional processes $w$, then $\tilde h$ is the unique optional projection of $h$. If $h$ is an integrand of class $\cP$ and there exists a predictable integrand $\hat h$ such that 
\[
\hat h(w)=\pp [h(w)]\quad
\]
for all bounded predictable processes $w$, then $\hat h$ is the unique predictable projection of $h$.
\end{lemma}

Given $K>0$, an integrand $h$ is said to be {\em $K$-Lipschitz} if, for all $(\omega,t)$ for which $h_t(\cdot,\omega)$ takes a value less than $+\infty$ somewhere, we have that $h_t(\cdot,\omega)$ is finite everywhere and 
\begin{align*}
|h_t(x,\omega)-h_t(y,\omega)|\le K|x-y|\quad\text{for all $x\in\reals^d$ and $y\in\reals^d$.}
\end{align*}

\begin{lemma}\label{lem:cp2}
Let $K>0$. The optional projection of a $K$-Lipschitz integrand of class $\cD$ is a $K$-Lipschitz integrand of class $\cD$.  The predictable projection of a $K$-Lipschitz integrand of class $\cP$ is a $K$-Lipschitz integrand of class $\cP$.
\end{lemma}
\begin{proof}
The existence and uniqueness of the projections follow from Theorem~\ref{thm:cp}. Let $h$ be a $K$-Lipschitz integrand of class $\cD$. Since the optional projection of a $\T$-integrable $m$ is $\T$-integrable, $\op h$ is of class $\cD$. Assume that the domain of
\[
\tilde \Gamma_t(\omega)=\{(x,y)\in\reals^d\times\reals^d \mid +\infty = \op h_t(x,\omega)>\tilde\alpha>\op h_t(y,\omega)\}
\]
is nonevanescent for some $\tilde \alpha\in\reals$.  Then, by Theorem~\ref{thm:ost}, there exists $\tilde\tau\in\T$ and bounded optional processes $\tilde v$ and $\tilde w$ such that
\[
\begin{split}
E[\one_{\{{\tilde\tau}<\infty\}}\tilde\alpha] &> E[\one_{\{{\tilde\tau}<\infty\}}\op h_{\tilde\tau}(\tilde w_{\tilde\tau})]\\
&=E[\one_{\{{\tilde\tau}<\infty\}}h_{\tilde\tau}(\tilde w_{\tilde\tau})]\\
&\ge E[\one_{\{{\tilde\tau}<\infty\}}(h_{\tilde\tau}(\tilde v_{\tilde\tau})-K|\tilde v_{\tilde\tau}-\tilde w_{\tilde\tau}|)]\\
&= E[\one_{\{{\tilde\tau}<\infty\}}(\op h_{\tilde\tau}(\tilde v_{\tilde\tau})-K|\tilde v_{\tilde\tau}-\tilde w_{\tilde\tau}|)]\\
&=+\infty,
\end{split}
\]
which is a contradiction. Thus, if $\op h_t(\cdot,\omega)$ takes a finite value somewhere, $\op h_t(\cdot,\omega)$ is finite everywhere. Now, assume that the domain of
\[
\Gamma_t(\omega)=\{(x,y)\in\reals^d\times\reals^d \mid \alpha > \op h_t(x,\omega)>\op h_t(y,\omega)+K|x-y|\}
\]
is nonevanescent for some $\alpha\in\reals$. Then, by Theorem~\ref{thm:ost}, there exists $\tau\in\T$ and bounded optional processes $v$ and $w$ such that
\[
\begin{split}
E [h_\tau(v_\tau)\one_{\{\tau<\infty\}}]&=E[\op h_\tau(v_\tau)\one_{\{\tau<\infty\}}]\\
&>E[\op h_\tau(w_\tau)\one_{\{\tau<\infty\}} +K|v_\tau-w_\tau|\one_{\{\tau<\infty\}}]\\
&=E[h_\tau(w_\tau)\one_{\{\tau<\infty\}} + K|v_\tau-w_\tau|\one_{\{\tau<\infty\}}]
\end{split}
\]
which is a contradiction. The claim for the predictable projection follows similarly from Theorem~\ref{thm:pst}.
\end{proof}

The following theorem is our main result on projections of normal integrands. It is a direct analog of the result in \cite{chs3} on conditional expectations of normal integrands.

\begin{theorem}\label{thm:eup}
The optional projection of a normal integrand of class $\cD$ is a normal integrand of class $\cD$. The predictable projection of a normal integrand of class $\cP$ is a normal integrand of class $\cP$.
\end{theorem}
\begin{proof}
Let $h$ be a normal integrand of class $\cD$ and $(B^i)_{i=1}^\infty$ the sequence of increasing open sets from the definition of class $\cD$. We define
\[
h^{\nu,i}_t(x,\omega)=\inf_{x'\in\reals^d}\{h_t(x',\omega)+\delta_{\bar B^i}(x')+\nu\|x-x'\|\},
\]
where $\bar B^i$ is the closure of $B^i$. Each $h+\delta_{\bar B^i}$ is a normal integrand, so, by \cite[Example 9.11]{rw98}, $h^{\nu,i}_t(\cdot,\omega)$ is Lipschitz and $h^{\nu,i}$ increases to $h+\delta_{\bar B^i}$. By \cite[p. 665]{rw98}, each $h^{\nu,i}$ is thus a normal Lipschitz integrand that is evidently of class $\cD$. By Lemma~\ref{lem:cp2}, $^o h^{\nu,i}$ is a normal integrand of class $\cD$. Since the sequence $(h^{\nu,i})_{\nu=1}^\infty$ is increasing, we get from Lemma~\ref{lem:cp1} and from \cite[Proposition 14.44]{rw98} that $k^i:=\sup_\nu  {^oh}^{\nu,i}$ is a normal integrand of class $\cD$. We will prove that 
\[
k_t(x,\omega) :=\inf_i k^i_t(x,\omega)
\]
is the optional projection of $h$.

Using Theorem~\ref{thm:mon}, we get $k^i=\op(h+\delta_{\bar B_i})$ so that $\one_{B^i}k^i=\one_{B^i}\op h=\one_{B^i}k^j$ for every $i\le j$. Since the sequence $k^i$ is nonincreasing, we thus see that  $\one_{B^i}k=\one_{B^i}k^i$ and that $k_t(\cdot,\omega)$ is lower semicontinuous which in conjunction with \cite[Proposition 14.44]{rw98} implies that $k$ is a normal integrand,
\[
\one_{B^i}k=\one_{B^i}k^i=\one_{B^i}\op h\quad\text{for all }i,
\]
$k$ is of class $\cD$, and that $k=\op h$. The predictable case is proved similarly.
\end{proof}

Recall that a real-valued function $h$ on $\Omega\times\reals_+\times\reals^d$ is called {\em Carath\'eodory integrand} if $h_t(\cdot,\omega)$ is continuous outside an evanescent set and $(\omega,t)\mapsto h_t(x,\omega)$ is measurable for all $x\in\reals^d$. By \cite[Example 14.29]{rw98}, a Carath\' eodory integrand is a normal integrand. We say that an integrand $h$ is of class $(C')$ if $h$ and $-h$ are of class $\cD$. Likewise, we say that an integrand $h$ is of class $(C_p')$ if $h$ and $-h$ are of class $\cP$.
\begin{corollary}
The optional projection of a Carath\'eodory integrand of class $(C')$ is a Carath\'eodory integrand of class $(C')$. The predictable projection of a Carath\'eodory integrand of class $(C_p')$ is a Carath\'eodory integrand of class $(C_p')$.
\end{corollary}
\begin{proof}
By Theorem~\ref{thm:eup}, $\op h$ and $\op [{-h}]$ are normal integrands of class $\cD$. We have $\op [{-h}(w)]=-\op [h(w)]$ for every bounded optional $w$, so Lemma~\ref{lem:bddc} gives that $\op [{-h}]=-\op h$. Thus both $\op h$ and $-\op h$ are normal integrand of class $\cD$ which proves the claim. The predictable case is proved similarly.
\end{proof}

The following corollary extends Theorem~\ref{thm:eup} beyond the class $\cD$.

\begin{corollary}\label{cor:eup}
Assume that $h$ is a normal integrand and let
\[
\bar h_t(x,\omega)=h_t(B_t(\omega)x+b_t(\omega),\omega),
\]
where $B$ is an optional invertible matrix-valued process and $b$ is an optional $\reals^d$-valued process. If $h$ is of class $\cD$, then $\op {\bar h}$ is a normal integrand given by 
\[
(\op {\bar h})_t(x,\omega)=(\op h)_t(B_t(\omega)x+b_t(\omega),\omega).
\]
If $h$ is of class $\cP$ and $B$ and $b$ are predictable, then $\pp {\bar h}$ is a normal integrand given by
\[
(\pp {\bar h})_t(x,\omega)=(\pp h)_t(B_t(\omega)x+b_t(\omega),\omega).
\]
\end{corollary}
\begin{proof}
By \cite[Proposition 14.45]{rw98}, $\bar h$ is a normal integrand. Let $\H$ be the set of optional processes $w$ for which $Bw+b$ is bounded, and let $\tilde w$ be an optional process. The set $A^\nu:=\{|B\tilde w+b|\le \nu\}$ is nonevanescent for $\nu$ large enough, and $w^\nu=1_{A^\nu}\tilde w-1_{(A^\nu)^C}B^{-1}b$ is optional with $Bw^\nu+b$ bounded. Therefore $\H$ is an $\O$-discretely dense set, and it is evidently contained in $\Lambda^\O_{\bar h}$. By Lemma~\ref{lem:bddc},
\[
\op h(Bw+b)=\op[h(Bw+b)]
\]
for every optional $w\in\H$, so Lemma~\ref{lem:dds} implies that
\[
\tilde h_t(x,\omega) :=(\op h)_t(B_t(\omega)x+b_t(\omega),\omega)
\]
is the optional projection of $\bar h$. Since $\op h$ is a normal integrand by Theorem~\ref{thm:eup}, $\tilde h$ is a normal integrand by \cite[Proposition 14.45]{rw98}. The predictable case is proved similarly.
\end{proof}

\section{Projections of convex-valued processes}\label{sec:svsp}

Let $S:\Omega\rightrightarrows\reals^d$ be a random closed set. We denote the almost sure selections of $S$ by $L^0(S)$. If there exists a random set $\hat S$ that is the smallest $\G$-measurable random closed set $\tilde S$ such that $E^\G\eta\in L^0(\tilde S)$ for every integrable $\eta\in L^0(S)$, then $\hat S$ is known as the {\em $\G$-conditional expectation} of $S$ and it is denoted by $E^\G S$. We refer to \cite{cv77,hu77,hes91,wan01,mol5} for equivalent formulations. Here we need the fact that if $S$ is {\em integrable} in the sense that it has an integrable selection, then the $\G$-conditional expectation of $S$ exists and it is unique up to a $P$-null set; see, e.g., \cite[Theorem 1.4]{wan01}.

\begin{definition}
Let $\Gamma$ be a measurable closed-valued stochastic process. If there exists a set-valued stochastic process $\hat\Gamma$ that is the smallest $\O$-measurable closed-valued stochastic process $\tilde\Gamma$ such that $\op w\in L^0(\tilde\Gamma)$ for every $\T$-integrable $w\in L^0(\Gamma)$, then we call $\hat\Gamma$ the {\em optional projection} of $\Gamma$ and denote it by $\hat\Gamma=\op\Gamma$.
\end{definition}

\begin{definition}
Let $\Gamma$ be a measurable closed-valued stochastic process. If there exists a set-valued stochastic process $\hat\Gamma$ that is the smallest $\P$-measurable closed-valued stochastic process $\tilde\Gamma$ such that $\pp w\in L^0(\tilde\Gamma)$ for every $\T_p$-integrable $w\in L^0(\Gamma)$, then we call $\hat\Gamma$ the {\em predictable projection} of $\Gamma$ and denote it by $\hat\Gamma=\pp\Gamma$.
\end{definition}

These definitions are consistent with the general theory of stochastic processes in the sense that when a set-valued process is single-valued, they reduce to the ordinary optional and predictable projection. They are also analogous to the single-valued case in the following way. We say that a set-valued process is {\em $\T$-integrable} or {\em $\T_p$-integrable} if it has a $\T$-integrable or $\T_p$-integrable selection, respectively. 

\begin{lemma}\label{lem:bst}
Let $\Gamma$ and $\hat\Gamma$ be measurable closed-valued stochastic processes. If $\Gamma$ is $\T$-integrable and $\hat\Gamma$ is optional, then $\hat\Gamma=\op\Gamma$ if and only if $\hat\Gamma_\tau=E^{\F_{\tau}} \Gamma_\tau$ for every bounded $\tau\in\T$. If $\Gamma$ is $\T_p$-integrable and $\hat\Gamma$ is predictable, then $\hat\Gamma=\pp\Gamma$ if and only if $\hat\Gamma_\tau=E^{\F_{\tau-}}\Gamma_\tau$ for every bounded $\tau\in\T_p$.
\end{lemma}
\begin{proof}
To prove the optional case, let $w$ be a $\T$-integrable selection of $\Gamma$. It suffices to prove that $\hat\Gamma$ is not the optional projection of $\Gamma$ if and only if there is a bounded $\tau\in\T$ such that $\hat\Gamma_\tau$ is not the $\F_\tau$-conditional expectation of $\Gamma_\tau$.

That $\hat\Gamma$ is not the optional projection of $\Gamma$ means that, (i) there is $v\in L^0(\hat\Gamma)$ such that $\op v\notin L^0(\widehat{\Gamma})$, or, (ii) there exists optional $\tilde{\Gamma}$ s.t. $v\in L^0(\Gamma)\implies \op v\in L^0(\tilde{\Gamma})$ and the strict inclusion $\tilde\Gamma\subset\hat\Gamma$ holds outside an evanescent set.

Let $\tau\in\T$ be bounded. That $\hat\Gamma_\tau$ is not the $\F_\tau$-conditional expectation of $\Gamma_\tau$ means that, (i') there exists $\eta\in L^0(\Gamma_\tau)$ s.t. $E^{\F_\tau}\eta\notin L^0(\hat{\Gamma}_\tau)$, or, (ii') there exists $\F_\tau$-measurable $\tilde{S}$ s.t. $\eta\in L^0(\Gamma_\tau)\implies E^{\F_\tau}\eta\in L^0(\tilde{S})$ and the strict inclusion $\tilde{S}\subset\hat{\Gamma}_\tau$ holds outside a $P$-null set.

For a bounded $\tau\in\T$, we have that $v\in L^0(\hat\Gamma)$ and $\eta\in L^0(\hat\Gamma_\tau)$ are in one-to-one correspondence via the mappings 
\begin{equation}\label{eq:corres}
\begin{split}
v&\mapsto v_\tau,\\
\eta&\mapsto\eta 1_{\{\graph\tau\}}+w1_{\{\graph\tau\}^C}.
\end{split}
\end{equation}
Thus it follows from Lemma~\ref{lem:inc} and $E^{\F_\tau}v_\tau=\op v_\tau$ that (i) is equivalent to that (i') holds for some bounded $\tau$. 

Similarly, for a bounded $\tau\in\T$, we have that $\tilde\Gamma$ in (ii) and $\tilde S$ in (ii') are in one-to-one correspondence via the mappings
\begin{align*}
\tilde\Gamma&\mapsto \tilde\Gamma_\tau,\\
\tilde S&\mapsto \tilde{S}1_{\{\graph\tau\}}+\hat{\Gamma}1_{\{\graph\tau\}^C}.
\end{align*}
Thus, it follows from Lemma~\ref{lem:inc} and \eqref{eq:corres} that (ii) is equivalent to that (ii') holds for some bounded $\tau$. The predictable case is proved similarly.
\end{proof}

Let $f:\Omega\times\reals^d\rightarrow\ereals$ be a normal $\F$-integrand on $\reals^d$. The $\G$-conditional expectation of the epigraphical mapping $\epi f$ is also an epigraphical mapping of some normal integrand whenever $\epi f$ has an integrable selection; see \cite[p. 136 and 140]{tru91}. We define the {\em $\G$-conditional epi-expectation} of $f$ as the integrand $^\G f$ whose epigraphical mapping is the $\G$-conditional expectation of the epigraphical mapping of $f$. Some authors have called this the conditional expectation of $f$. In general, conditional epi-expectation and conditional expectation do not coincide (consider, e.g., a trivial $\sigma$-algebra $\G$ and $f(x,\omega)=\delta_{\{\xi(\omega)\}}(x)$ for a non-constant integrable random variable $\xi$), so we introduced a new term, conditional epi-expectation, to distinguish these two notions. 

\begin{definition}
Let $h$ be a normal integrand on $\reals^d$. If there exists an integrand $\hat h$ whose epigraph is the optional projection of $\epi h$, then we call $\hat h$ the {\em optional epi-projection} of $h$ and denote it by $\hat h= {^\O h}$. 
\end{definition}
\begin{definition}
Let $h$ be a normal integrand on $\reals^d$. If there exists an integrand $\hat h$ whose epigraph is the predictable projection of $\epi h$, then we call $\hat h$ the {\em predictable epi-projection} of $h$ and denote it by $\hat h= {^\P h}$. 
\end{definition}

In this section we study the existence and uniqueness of optional and predictable projections of convex-valued stochastic processes. Our proof is based on conjugacy arguments and on our existence results for projections of normal integrands. We point out that the techniques in \cite{hu77,wan01}, that are used to obtain the existence of conditional expectations of  random nonconvex sets, do not extend to the continuous time setting as such. Briefly, both approaches use the reference measure on the underlying space, but we do not have such measure on $\Omega\times\reals_+$. We leave it as an open question how to prove the existence of projections in the nonconvex case.

We assume from now on that $h$ is a convex normal integrand on $\reals^d$. By \cite[Theorem 14.50]{rw98}, the {\em conjugate integrand}
\[
h^*_t(y,\omega)=\sup_{x\in\reals^d}\{x\cdot y- h_t(x,\omega)\}
\]
of $h$ is a convex normal integrand on $\reals^d$. In other words, $h^*$ is an $(\omega,t)$-wise convex conjugate of $h$.
\begin{theorem}\label{thm:con}
Assume that $h$ is a convex normal integrand. If $h^*(v)^+$ is $\T$-integrable for some $\T$-integrable $v$, then $h$ and $\op h$ are convex normal integrands of class $\cD$. If $h^*(v)^+$ is $\T_p$-integrable for some $\T_p$-integrable $v$, then $h$ and $\pp h$ are convex normal integrands of class $\cP$.
\end{theorem}
\begin{proof}
The Fenchel inequality
\[
h_t(x,\omega)+h^*_t(v_t(\omega),\omega)\ge x \cdot v_t(\omega)
\]
implies that $h$ is of class $\cD$. By Theorem~\ref{thm:eup}, $\op h$ exists and it is a normal integrand of class $\cD$. To prove that $\op h$ is convex, let $\alpha\in(0,1)$ and
\begin{align*}
\tilde h_t(x^1,x^2,\omega) &:=h_t(\alpha x^1+(1-\alpha)x^2,\omega),\\
\bar h_t(x^1,x^2,\omega) &:=\alpha h_t(x^1,\omega)+(1-\alpha)h_t(x^2,\omega).
\end{align*}
Both are normal integrands by \cite[Proposition 14.44 and Proposition 14.45]{rw98}. By Lemma~\ref{lem:bddc}, 
\begin{align*}
\op{\tilde h}_t(x^1,x^2,\omega) &=\op h_t(\alpha x^1+(1-\alpha)x^2,\omega),\\
\op{\bar h}_t(x^1,x^2,\omega) &=\alpha \op h_t(x^1,\omega)+(1-\alpha)\op h_t(x^2,\omega).
\end{align*}
Since $h$ is convex, $\tilde h\le \bar h$, so Lemma~\ref{lem:cp1} implies that $\op{\tilde h}\le \op{\bar h}$. Thus $\op h$ is convex as well. The predictable case is proved similarly.
\end{proof}

Recall that the recession function of a proper lower semicontinuous convex function $g$ is given by
\[
g^\infty(x)=\sup_{\lambda>0} \frac{g(\lambda x+\bar x)-g(\bar x)}{\lambda}
\]
which is independent of the choice of $\bar x\in\dom g$ \cite[Theorem 8.5]{roc70a}. For a convex normal integrand $h$, we denote the $(\omega,t)$-wise recession function of $h$ by $h^\infty$. By \cite[Exercise 14.54]{rw98}, $h^\infty$ is a convex normal integrand.
\begin{theorem}\label{thm:rec}
Assume that $h$ is a convex normal integrand. If $h(w)^+$ and $h^*(v)^+$ are $\T$-integrable for some optional $w$ and some $\T$-integrable $v$, then $\op [h^\infty]=[\op h]^\infty$. If $h(w)^+$ and $h^*(v)^+$ are $\T_p$-integrable for some predictable $w$ and some $\T_p$-integrable $v$, then $\pp [h^\infty]=[\pp h]^\infty$. 
\end{theorem}
\begin{proof}
By Theorem~\ref{thm:con}, $h$ is of class $\cD$. By \cite[Theorem 23.1]{roc70a}, the difference quotients
\[
h^\lambda_t(x,\omega):=\frac{h_t(\lambda x+w_t(\omega),\omega)-h_t(w_t(\omega),\omega)}{\lambda}
\]
define a nondecreasing sequence of integrands $(h^\lambda)_{\lambda=1}^\infty$. As in the proof of Corollary~\ref{cor:eup}, the set $\H$ of optional processes $\tilde w$, for which $\tilde w+w$ is bounded, is an $\O$-discretely dense set such that $h^1(\tilde w)^-$ is $\T$-integrable for each $\tilde w\in\H$.  Thus the result follows from Theorem~\ref{thm:mon} and Corollary~\ref{cor:eup}. The predictable case is proved similarly.
\end{proof}

For a convex normal $\F$-integrand $f:\Omega\times\reals^d\rightarrow\ereals$ on $\reals^d$, \cite[Theorem 2.1.2 and Corollary 2.1.1.1]{tru91} imply that
\begin{align}\label{eq:ddc}
E^{\G}\epi f=\epi ( [E^\G (f^*)]^* ) 
\end{align}
whenever there exists $\xi\in L^1(\F;\reals^d)$ and $\eta\in L^0(\G;\reals^d)$ such that $f(\xi)^+\in L^1$ and $f^*(\eta)^+\in L^1$. Note that \eqref{eq:ddc} means
\[
^\G f= [E^\G (f^*)]^*.
\]
Under analogous assumptions, the next theorem shows that epi-projections and projections of normal integrands are dual operations in the same sense. The conditions in the theorem cannot be omitted in general, as demonstrated after the theorem. As to their role for sufficiency, we refer to the proof of  \cite[Theorem 2.1.2]{tru91}.

\begin{theorem}\label{thm:epip}
Assume that $h$ is a convex normal integrand. If $h(w)^+$ and $h^*(v)^+$ are $\T$-integrable for some $\T$-integrable $w$ and some optional $v$, then $^\O h$ is given by 
\begin{equation}\label{eq:oo}
	{^\O h}=(\op(h^*))^*.
\end{equation}
If  $h(w)^+$ and $h^*(v)^+$ are $\T_p$-integrable for some $\T_p$-integrable $w$ and some predictable $v$, then $^\P h$ is given by
\begin{equation}\label{eq:op}
	{^\P h}=(\pp(h^*))^*.
\end{equation}
\end{theorem}
\begin{proof}
By Theorem~\ref{thm:con}, $\op (h^*)$ exists and it is a convex normal integrand. We will show that $[\op (h^*)]^*$ is the optional epi-projection of $h$. By Lemma~\ref{lem:bst}, $\epi [\op (h^*)]^*$ is the optional projection of $\epi h$ if and only if $\epi ([\op (h^*)]^*_\tau)$ is the $\F_\tau$-conditional expectation of $\epi (h_\tau)$ for every bounded $\tau\in\T$.

For a bounded $\tau\in\T$,
\[
(y,\omega)\mapsto h^*_{\tau(\omega)}(y,\omega)
\]
is a convex normal $\F$-integrand on $\reals^d$. We have that bounded $\F_\tau$-measurable random variables form an $\F_\tau$-discretely dense subset of $\Lambda^{\F_\tau}_{h^*_\tau}$ and bounded optional processes form an $\O$-discretely dense subset of $\Lambda^{\O}_{h^*}$. Moreover, $\eta$ is bounded and $\F_\tau$-measurable if and only if there exists bounded and optional process $\tilde{w}$ such that $\tilde{w}_\tau=\eta$, so $E^{\F_\tau} (h^*_\tau(\eta))=E^{\F_\tau} (h^*_\tau(\tilde{w}_\tau))=\op [h^*]_\tau(\tilde{w}_\tau)$ a.s., i.e., $E^{\F_\tau} (h^*_\tau)=\op [h^*]_\tau$. By \eqref{eq:ddc},
\[
E^{\F_\tau}(\epi h_\tau)= \epi ([E^{\F_\tau} (h^*_\tau)]^*)= \epi ([\op [h^*]_\tau]^*)= \epi ([\op [h^*]]_\tau^*),
\]
which proves that $[\op (h^*)]^*$ is the optional epi-projection of $h$. The predictable case is proved similarly.
\end{proof}

The assumptions in Theorem~\ref{thm:epip} cannot be dropped in general. For example, \eqref{eq:oo} and \eqref{eq:op} fail if $h_t(x,\omega)=\delta_{\{v_t(\omega)\}}(x)$ or $h_t(x,\omega)=\delta_{\{0\}}(x)-\alpha_t(\omega)$, where $v$ is real-valued, nonnegative and predictable that is not $\T_p$-integrable, and $\alpha_\tau$ is real-valued, nonnegative, nonintegrable and independent of $\F_\tau$ for each $\tau\in\T$.

The formulas in the following result reduce to Jensen's inequalities for optional and predictable integrands, since then $^\O h= h$ and $^\P h=h$, respectively, by Theorem~\ref{thm:epip}.
\begin{theorem}
Assume that $h$ is a convex normal integrand. If $h(\bar w)^+$ and $h^*(\bar v)$ are $\T$-integrable for some $\T$-integrable $\bar w$ and bounded optional $\bar v$, then
\begin{align}\label{eq:gjino}
\begin{split}
 ^\O h(\op w)\le \op [h(w)]
\end{split}
\end{align}
for every $\T$-integrable $w$. If $h(\bar w)^+$ and $h^*(\bar v)$ are $\T_p$-integrable for some $\T_p$-integrable $\bar w$ and bounded predictable $\bar v$, then
\begin{align}\label{eq:gjinp}
\begin{split}
 ^\P h(\pp w)\le \pp [h(w)]
\end{split}
\end{align}
for every $\T_p$-integrable $w$. 
\end{theorem}
\begin{proof}
It suffices to show that, for every $\tau\in\T$ and $A\in\F_\tau$,
\begin{align*}
E[\one_A\one_{\{\tau<\infty\}}{^\O h}_\tau(\op w_\tau)]\le E[\one_A\one_{\{\tau<\infty\}}h_\tau(w_\tau)].
\end{align*}
Without loss of generality, we may assume that the right side is not $+\infty$. By \cite[Theorem 3.9]{hwy92}, $\tau_A:=\tau+\delta_A$ belongs to $\T$. Let $\tilde w:=1_{\graph {\tau_A}}w+1_{(\graph{\tau_A})^C}\bar w$ so that $(\tilde w,h(\tilde w))$ is a $\T$-integrable selection of $\epi h$. Indeed, $h(\tilde w)^+$ is $\T$-integrable by assumptions, and $h(\tilde w)^-$ is $\T$-integrable by the Fenchel inequality and by the assumptions on $\bar v$. Thus $^\O h(\op{\tilde w})\le\op[h(\tilde w)]$ by the definition of optional epi-projection, which implies the claim. The predictable case is proved similarly, where, by \cite[Theorem 3.29]{hwy92}, $\tau_A\in\T_p$ whenever $\tau\in\T_p$ and $A\in\F_{\tau-}$.
\end{proof}

The following theorem is our main result on optional and predictable projections of set-valued processes. Recall that a set-valued process is $\T$-integrable or $\T_p$-integrable if it has a $\T$-integrable or $\T_p$-integrable selection, respectively. 
\begin{theorem}\label{thm:eusp}
Assume that $\Gamma$ is convex-valued stochastic process. If $\Gamma$ is $\T$-integrable, then $\op\Gamma$ is given uniquely by
\[
 \delta_{\op\Gamma}=[\op(\delta_\Gamma^*)]^*.
\]
If $\Gamma$ is $\T_p$-integrable, then $\pp\Gamma$ is given uniquely by
\[
 \delta_{\pp\Gamma}=[\pp(\delta_\Gamma^*)]^*.
\]
\end{theorem}
\begin{proof}
By Theorem~\ref{thm:epip} and Theorem~\ref{thm:rec}, ${^\O \delta_\Gamma}=[\op(\delta_\Gamma^*)]^*$, where $\op(\delta_\Gamma^*)=[\op(\delta_\Gamma^*)]^\infty$, so ${^\O \delta_\Gamma}$ is a conjugate of an optional positively homogeneous convex normal integrand, and consequently it is an indicator function of a unique optional closed convex-valued stochastic process $\op\Gamma$ \cite[Theorem 8.24]{rw98}. The predictable case is proved similarly.
\end{proof}

The following result is a simple extension of \cite[Lemma 3.3]{bp87} which was formulated for bounded processes.
\begin{corollary}\label{cor:jii}
If $\Gamma$ is an optional closed convex-valued stochastic process and $w\in L^0(\Gamma)$ is $\T$-integrable, then $\op w\in L^0(\Gamma)$. If $\Gamma$ is a predictable closed convex-valued stochastic process and $w\in L^0(\Gamma)$ is $\T_p$-integrable, then $\op w\in L^0(\Gamma)$. 
\end{corollary}
\begin{proof}
Assume that $\Gamma$ is optional. Since $\delta^*_{\Gamma}(0)=0$, Theorem~\ref{thm:epip} gives $^\O \delta_{\Gamma}=\delta_{\Gamma}$, so $\op w\in L^0(\Gamma)$ by \eqref{eq:gjino}. The predictable case is proved similarly.
\end{proof}

\section{Projections of convex-valued integrands}\label{sec:svi}

Let $F$ be a {\em set-valued $\S$-integrand} from $\Omega\times\reals_+\times\reals^m$ to $\reals^n$ in the sense that 
\[
\Gamma^F_t(\omega) :=\{(x,u)\in\reals^m\times\reals^n \mid  u\in F_t(x,\omega)\}
\]
defines an $\S$-measurable set-valued stochastic process. It is called {\em closed-valued} if, outside an evanescent set, $F_t(\cdot,\omega)$ is a closed-valued mapping. Likewise, $F$ is called {\em inner semicontinuous} (isc) if,  outside an evanescent set, $F_t(\cdot,\omega)$ is isc. Moreover, $F$ is called {\em optional} if $\Gamma^F$ is so. The properties {\em convex-valued, graph-convex, outer semicontinuous}, and {\em predictable} are defined similarly. Recall that a set-valued mapping from a topological space to another is {\em inner semicontinuous} (isc) if preimages of open sets are open, and it is {\em outer semicontinuous} (osc) if its graph is closed.

Analogs of optional and predictable projections for a class of isc integrands are given after the following preparatory lemma. More general cases without continuity assumptions are left for further research.

\begin{lemma}\label{lem:isc}
A convex-valued mapping $F$ from $\Omega\times\reals_+\times\reals^m$ to $\reals^n$ is an isc convex-valued $\S$-integrand if and only if
\[
h_t(x,y,\omega):=\delta_{F_t(x,\omega)}^*(y)
\]
is a normal $\S$-integrand on $\reals^m\times\reals^n$.
\end{lemma}
\begin{proof}
By \cite[Lemma A.2]{bv88}, inner semicontinuity of $F$ implies that $h_t(\cdot,\cdot,\omega)$ is jointly lsc. We have $h_t(x,y,\omega)=-\inf_u p_t(x,y,u,\omega)$ for
\[
p_t(x,y,u,\omega)=-u\cdot y+\delta_{F_t(x,\omega)}(u).
\]
Using $\S$-measurability of $\Gamma^F$, one may verify straight from the definition that $\epi p$ is $\S$-measurable mapping from $\Omega\times\reals_+$ to $\reals^m\times\reals^n\times\reals^n\times\reals$. Since the projection mapping from $\reals^m\times\reals^n\times\reals^n\times\reals$ to $\reals^m\times\reals^n\times\reals$ is isc (it is single-valued and continuous), we get measurability of $\hypo h$ from measurability of $\epi p$ and \cite[Theorem 14.13(a)]{rw98}. Here
\[
(\hypo h)_t(\omega):=\{(x,y,\alpha)\mid h_t(x,y,\omega)\ge\alpha\}.
\]
Evidently, the mapping defined by $\Gamma_t(\omega):=\{(x,y,\alpha)\mid h_t(x,y,\omega)>\alpha\}$ is measurable as well. Since  $\Gamma^{-1}(O)=\{(t,\omega)\mid \epi h_t(\omega)\subset O\}^C$ for any open $O$, we get measurability of $\epi h$ from \cite[Theorem 14.3(h)]{rw98}. The converse can be proved similarly using the orthogonal projection of the epigraph of
\[
\tilde p_t(x,y,u,\omega)=-u\cdot y+h_t(x,y,\omega)
\]
and the fact that \cite[Theorem 14.3(h)]{rw98} is applicable here as well. Indeed, the proof of part (h) does not use closed-valuedness of the set-valued mapping.
\end{proof}  

Note that $F(w)$ defines an $\S$-measurable set-valued stochastic process for every $\S$-measurable stochastic process $w$. This follows from \cite[Theorem 14.13(a)]{rw98} by choosing $M$ as the projection from $\reals^n\times\reals^m$ to $\reals^n$ and $S_t(\omega)=[\{w_t(\omega)\}\times\reals^n]\cap\Gamma^F_t(\omega)$. We denote 
\begin{align*}
\Lambda^\O_F &=\{w\mid w\text{ is optional and there exists $\T$-integrable }u\in L^0(F(w))\},\\
\Lambda^\P_F &=\{w\mid w\text{ is predictable and there exists $\T_p$-integrable }u\in L^0(F(w))\}.
\end{align*}

\begin{theorem}
Let $F$ be an inner-semicontinuous closed convex-valued integrand and assume that there exists a stochastic process $m$ such that
\[
\inf \{|u|\mid u\in F_t(x,\omega)\}\le m_t(\omega)|x|\quad\forall (\omega,t,x)\text{ such that } F_t(x,\omega)\ne\emptyset.
\]
If $m$ is $\T$-integrable, then there exists a unique optional inner-semicontinuous closed convex-valued integrand $\op F$ such that 
\[
[\op F](w) =\op[F(w)]\quad w\in\Lambda^\O_F.
\]
If $m$ is $\T_p$-integrable, then there exists a unique predictable inner-semicontinuous closed convex-valued integrand $\pp F$ such that 
\[
[\pp F](w) =\op[F(w)]\quad w\in\Lambda^\P_F.
\]
\end{theorem}
\begin{proof}
We define
\[
h_t(x,y,\omega)=\delta_{F_t(x,\omega)}^*(y).
\]
We get from the existence of $m$ that 
\[
h_t(x,y,\omega)\ge - m_t(\omega)|x||y|,
\]
which in conjunction with  Lemma~\ref{lem:isc} implies that $h$ is a normal integrand of class $\cD$. Thus, by Theorem~\ref{thm:eup}, there exists a unique optional projection $\op h$ of $h$ that is a normal integrand of class $\cD$.  

The projection $\op h$ inherits the convexity and positively homogeneity of $h$ in the $y$-argument. Both properties can be shown by defining, for $\alpha\in\reals_+$,
\begin{align*}
\tilde h_t(x,y^1,y^2,\omega) &:=h_t(x,\alpha y^1+(1-\alpha)y^2,\omega),\\
\bar h_t(x,y^1,y^2,\omega) &:=\alpha h_t(x,y^1,\omega)+(1-\alpha)h_t(x,y^2,\omega)
\end{align*}
and proceeding as in the proof of Theorem~\ref{thm:con}. Convexity follows by considering $\alpha\in(0,1)$ whereas positive homogeneity follows by considering $y^2=0$. Thus, $\op h_t(x,\cdot,\omega)$ is a support function of some closed convex set $\tilde F_t(x,\omega)$. By Lemma~\ref{lem:isc}, $\tilde F$ is an optional inner-semicontinuous closed convex-valued integrand. To finish the proof, we show that we may set $\op F:=\tilde F$.

Given a $w\in\Lambda^\O_F$, the existence of $\T$-integrable $u\in L^0(F(w))$ implies that 
\[
\bar h_t(y,\omega):=h_t(w_t(\omega),y,\omega) \ge u_t(\omega)\cdot y\quad\forall (\omega,t,y).
\]
By \cite[Proposition~14.45]{rw98}, $\bar h$ thus defines a normal integrand of class $\cD$. The above lower bound also gives that $(w,v)\in\Lambda^\O_h$ for every bounded optional $v$. Thus
\[
\op {\bar h}(v)= \op h(w,v)
\]
for every bounded optional $v$, which implies, by Theorem~\ref{thm:eusp}, that 
\[
\op[F(w)]=\tilde F(w).
\]
The predictable case is proved similarly.
\end{proof}

For set-valued integrands, the analogs of epi-projections are the projections of the associated graphical mappings. Indeed, if $F$ is the epigraphical mapping of a normal integrand, these definitions give exactly the epi-projections. The following theorem is an immediate corollary of Theorem~\ref{thm:eusp}.
\begin{theorem}
Let $F$ be an outer semicontinuous graph-convex integrand. If there exists $\T$-integrable $w\in\Lambda^\O_F$, then there exists an optional outer semicontinuous graph-convex integrand $^\O F$ such that $\Gamma^{{^\O F}}=\op \Gamma^F$ and it is given uniquely by
\[
 \delta_{\Gamma^{{^\O F}}}=[\op(\delta_{\Gamma^F}^*)]^*.
\]
If there exists $\T_p$-integrable $w\in\Lambda^\P_F$, then there exists a predictable outer semicontinuous graph-convex integrand $^\P F$ such that $\Gamma^{{^\P F}}=\pp \Gamma^F$ and it is given uniquely by
\[
 \delta_{\Gamma^{{^\P F}}}=[\pp(\delta_{\Gamma^F}^*)]^*.
\]
\end{theorem}

\subsubsection*{Acknowledgements}

The authors thank the anonymous referees whose comments and suggestions helped substantially to improve the article. In particular, the last section on set-valued integrands was inspired by their insights.

\bibliographystyle{plain}
\bibliography{sp}

\end{document}